\documentclass[letterpaper,10pt,journal]{IEEEtran}

\usepackage{graphics}
\usepackage{epsfig} 
\usepackage{amsmath}
\usepackage{amssymb}
\usepackage{dsfont}
\usepackage{enumerate}
\usepackage[tight]{subfigure}
\usepackage{multirow}
\usepackage{bm}
\usepackage{tikz}
\usepackage[noadjust]{cite}
\usepackage{amsthm}  
\usepackage{array,makecell}
\usepackage{cancel}

\usepackage{hyperref}

\usepackage[ruled,linesnumbered]{algorithm2e}

 \RestyleAlgo{ruled}

\SetKwInput{KwInit}{Initialise}
 \SetKwRepeat{Do}{do}{while}

\usetikzlibrary{arrows.meta}
\usetikzlibrary{positioning}
\usetikzlibrary{shadows,shadings,shapes.symbols, calc}

\newcommand{\F}{\mathcal{F}}
\newcommand{\Fone}{\mathcal{F}^A}

\newcommand{\Fthree}{\mathcal{F}^{Pol}}
\newcommand{\V}{\mathcal{V}}

\newcommand{\y}{\mathbf{y}}
\newcommand{\z}{\mathbf{z}}
\newcommand{\uu}{\mathbf{u}}

\newcommand{\detm}{\text{det}}

\title{Path-Based Conditions for the Identifiability of Non-additive Nonlinear Networks with Full Measurements}

\author{Renato Vizuete and Julien M. Hendrickx
\thanks{This work was supported by F.R.S.-FNRS via the \emph{KORNET} project, and by the \emph{SIDDARTA} Concerted Research Action (ARC) 
of the Fédération Wallonie-Bruxelles.}
\thanks{R.~Vizuete and J.~M.~Hendrickx are with ICTEAM institute, UCLouvain, B-1348, Louvain-la-Neuve, Belgium. R.~Vizuete is a FNRS Postdoctoral Researcher - CR.
{\tt\small renato.vizueteharo@uclouvain.be},
{\tt\small julien.hendrickx@uclouvain.be}\protect.}
}
 
\newcommand{\vertiii}[1]{{\left\vert\kern-0.25ex\left\vert\kern-0.25ex\left\vert #1 
    \right\vert\kern-0.25ex\right\vert\kern-0.25ex\right\vert}}

\newtheorem{definition}{Definition}
\newtheorem{assumption}{Assumption}
\newtheorem{theorem}{Theorem}

\newtheorem{proposition}{Proposition}
\newtheorem{lemma}{Lemma}
\newtheorem{remark}{Remark}
\newtheorem{example}{Example}

\newcommand{\abs}[1]{\left|#1\right|}
\newcommand{\norm}[1]{\abs{\abs{#1}}}

\newcommand{\R}{\mathbb R}

\newcommand{\NN}{\mathcal{N}}

\begin{document}

\maketitle
\thispagestyle{empty}

\begin{abstract}
We analyze the identifiability of nonlinear networks with non necessarily additive node dynamics, where the influence of in-neighbors is represented by a multivariate nonlinear function that cannot necessarily be separated into individual edge functions.
We consider the full measurement case (all the nodes are measured) and we introduce the notion of generic identifiability. Based on a generic nonlinear matrix associated with an unfolded digraph constructed from the network, we characterize the space of functions that satisfies the generic property. For directed acyclic graphs (DAGs) composed of analytic functions, we derive a sufficient condition for identifiability based on vertex-disjoint paths from excited nodes to the in-neighbors of each node in the network. Furthermore, for the class of polynomial functions, by using well-known results on algebraic varieties, we prove that the identifiability is impossible if the vertex-disjoint path condition is not satisfied. Finally, we show that this identifiability condition is not necessary for the additive nonlinear model, where the node function can be decomposed into a sum of edge-wise nonlinearities.
\end{abstract}

\begin{IEEEkeywords}
Network analysis and control, system identification.
\end{IEEEkeywords}

\section{Introduction}

Identifiability is a fundamental property in the identification of networked systems \cite{boccaletti2006complex,bullo2022lectures} that allows us to determine which nodes must be excited and measured to identify all the existing dynamics in the system. It is commonly assumed that the network topology is known, which allows obtaining conditions for identifiability based primarily on the structure of the network. When linear dynamics are located in the edges, the identifiability conditions have been fully characterized in the case of full excitation or full measurement \cite{weerts2018identifiability,hendrickx2019identifiability,vanwaarde2020necessary}. 
For more complex interactions \cite{pan2012reconstruction,dorfler2014synchronization,aalto2020gene,bizyaeva2023nonlinear}, most of the dynamics in the edges are nonlinear and the identifiability conditions turn out to be in many cases different from, and in some cases weaker than, those in the linear case due to the lower risk of ambiguities created by linear superposition 
\cite{vizuete2023nonlinear,vizuete2026nonlinear}.

Previous identifiability conditions for nonlinear networks have been derived for additive models. 
In the nonlinear case, a model that is not additive allows us to encompass richer behaviors, as in neural networks \cite{aggarwal2018neural}, opinion dynamics \cite{bizyaeva2023nonlinear}, interconnection of Wiener models \cite{bonassi2024structured}, complex networks \cite{zanudo2017structure}, among others. In this case, the dynamics of a node are represented by a multivariate nonlinear function of past values of the outputs of its in-neighbors, which cannot necessarily be decomposed into a sum of independent edge functions, and the identifiability problem is formulated in node dynamics \cite{vanelli2025local}.

In the full excitation case \cite{vizuete2023nonlinear,vizuete2026nonlinear}, the identifiability conditions are sufficient and necessary for any function in the specific class. However, for some graph topologies, it is possible that some identifiability conditions are valid except for a particular choice of parameters in the dynamics.  
These types of results are known as \emph{generic} \cite{van1991graph}, since they hold except on a zero measure set of parameters such as in the linear case \cite{hendrickx2019identifiability,mapurunga2023identifiability,legat2024identifiability}. However, in the nonlinear case, the functions belong to a space of infinite dimension, and a proper characterization of this notion of genericity is fundamental to determine what a zero measure set encompasses \cite{vanelli2025local}.

A preliminary version of this paper was presented in \cite{vizuete2024partial} where a sufficient condition for the identifiability of DAGs was derived in terms of vertex-disjoint paths when the model is additive nonlinear. Furthermore, it was shown that certain identifiability conditions might not hold for particular choices of functions in the edges, but a formal characterization of this notion of \emph{genericity} was missing. 
In this document, we provide a rigorous definition of genericity based on the nonlinear matrix $J_{H_v(G)}(\uu^v)$ that will play an important role in the proof of Theorem~\ref{thm:sufficiency}.
By contrast to \cite{vizuete2024partial}, in this work, we derive similar identifiability conditions which are sufficient for identifiability for a more general model that is non-additive. In addition, we fully characterize the notion of genericity that was missing in \cite{vizuete2024partial}. Finally, we show that the sufficient condition in \cite{vizuete2024partial} is not necessary for the additive model.

\section{Problem formulation}\label{sec:problem}
\subsection{Model class}

We consider a network characterized by a weakly connected digraph $G=(V,E)$ composed of a set of nodes $V=\{1,\ldots,n\}$ and a set of edges $E\subseteq V\times V$, where a directed edge from $j$ to $i$ is denoted as $(i,j)$. The output of each node $i$ in the network is given by the nonlinear dynamics
\begin{multline}\label{eq:general_nonlinear}
\hspace{-3mm}y_i^k\!=\!u_i^{k-1}+\Phi_i(y_1^{k-0},\ldots,y_1^{k-m_{i,1}},\ldots,y_{M_i}^{k-0},\ldots,y_{M_i}^{k-m_{i,M_i}}), \\ \text{for } y_1,\ldots,y_{M_i}\in\NN_i,  
\end{multline}
where the superscripts of the inputs and outputs denote the values at the specific time instants, $y_i^k$ is the output of the node $i$, $u_i^{k-1}$ is an arbitrary external excitation signal, the delays $m_{i,j}\in \mathbb{Z}^{\ge}$ are nonnegative, $\Phi_i$ is the node function associated with node $i$, $M_i$  is the number of in-neighbors of node $i$, and $\NN_i$ is the set of in-neighbors of node $i$. 

If a node $i$ is not excited, its corresponding excitation signal $u_i$ is set to zero. 
The model  \eqref{eq:general_nonlinear} corresponds to a generalized version of the nonlinear model in \cite{vizuete2023nonlinear,vizuete2024partial}, where the node function $\Phi_i$ is not necessarily additive.

Unlike the additive model in \cite{vizuete2023nonlinear,vizuete2026nonlinear,vizuete2024partial}, we will assume that the topology of the network $G$ defines the potential arguments of the function $\Phi_i$. If $(i,j)\in E$ then $\frac{\partial \Phi_i}{\partial y_j^{k-p}}\not\equiv 0$ for at least some $0\le p\le m_{i,j}$ (i.e., actually depends on nontrivial outputs of the in-neighbors). We do not consider multiple edges between two nodes, since they would be indistinguishable and hence unidentifiable.

\begin{assumption}\label{ass:full_excitation}
    The graph $G$ associated with the network is known, where the presence of an edge $(i,j)$ implies $\frac{\partial \Phi_i}{\partial y_j^{k-p}}\not\equiv 0$ for at least some $0\le p\le m_{i,j}$.
\end{assumption}

Assumption~\ref{ass:full_excitation} implies that we know if a node function $\Phi_i$ is influenced by the output of a specific node $y_j$.

\begin{figure}[!t]
    \centering
    \begin{tikzpicture}
    [fill fraction/.style n args={2}{path picture={
            \fill[#1] (path picture bounding box.south west) rectangle
            ($(path picture bounding box.north west)!#2!(path picture bounding box.north east)$);}},
roundnodes/.style={circle, draw=black!60, fill=black!25, very thick, minimum size=1mm},roundnode/.style={circle, draw=white!60, fill=white!5, very thick, minimum size=1mm},roundnodes4/.style={circle, draw=black!60, fill=black!25, fill fraction={white!30}{0.5}, very thick, minimum size=1mm}
]

\small

\node[roundnodes4](node1){1};
\node[roundnodes4](node2)[above=of node1,yshift=0.0cm,xshift=2.6cm]{2};
\node[roundnodes](node3)[right=of node1,yshift=0mm,xshift=1cm]{3};
\node[roundnodes](node4)[below=of node1,yshift=-0.0cm,xshift=2.6cm]{4};
\node[roundnodes](node5)[right=of node3,yshift=0mm,xshift=1cm]{5};

\node(phi5)[above=of node5,yshift=-1.05cm,xshift=0.95cm]{\footnotesize{$y_5\!=\!\Phi_5(y_2,y_3,y_4)$} };
\node(phi2)[above=of node2,yshift=-1.4cm,xshift=1.6cm]{\footnotesize{$y_2\!=\!\Phi_2(y_1,y_3)\!+\!u_2$} };
\node(phi4)[above=of node4,yshift=-1.6cm,xshift=1.3cm]{\footnotesize{$y_4\!=\!\Phi_4(y_1,y_3)$} };
\node(phi3)[above=of node3,yshift=-1.1cm,xshift=0.8cm]{\footnotesize{$y_3\!=\!\Phi_3(y_1)$} };
\node(phi1)[above=of node1,yshift=-1.0cm,xshift=0.0cm]{\footnotesize{$y_1\!=\!u_1$} };

\node[roundnode](u1)[above=of node1,yshift=-12mm,xshift=-12mm]{$u_1$};
\node[roundnode](u2)[above=of node2,yshift=-12mm,xshift=-12mm]{$u_2$};

\draw[-{Classical TikZ Rightarrow[length=1.5mm]}] (node1) to node [right,swap,yshift=2.5mm] {} (node2);
\draw[-{Classical TikZ Rightarrow[length=1.5mm]}] (node1) to node [right,swap,yshift=2.5mm] {} (node4);
\draw[-{Classical TikZ Rightarrow[length=1.5mm]}] (node1) to node [right,swap,yshift=2.5mm] {} (node3);
\draw[-{Classical TikZ Rightarrow[length=1.5mm]}] (node3) to node [right,swap,yshift=2.5mm] {} (node4);
\draw[-{Classical TikZ Rightarrow[length=1.5mm]}] (node2) to node [right,swap,yshift=2.5mm] {} (node5);
\draw[-{Classical TikZ Rightarrow[length=1.5mm]}] (node3) to node [right,swap,yshift=2.5mm] {} (node5);
\draw[-{Classical TikZ Rightarrow[length=1.5mm]}] (node4) to node [right,swap,yshift=2.5mm] {} (node5);
\draw[-{Classical TikZ Rightarrow[length=1.5mm]}] (node3) to node [right,swap,yshift=2.5mm] {} (node2);

\draw[gray,dashed,-{Classical TikZ Rightarrow[length=1.5mm]}] (u1) -- (node1);
\draw[gray,dashed,-{Classical TikZ Rightarrow[length=1.5mm]}] (u2) -- (node2);

\end{tikzpicture}

\normalsize

    \caption{Model of a network considered for the identification where all the nodes are measured (gray) and some nodes can be excited  (white and gray). The dynamics of each node is determined by a nonlinear function $\Phi_i$ of the outputs of the in-neighbors.}

    \label{fig:model}
\end{figure}

In this work, we restrict our attention to networks that do not contain any cycle (i.e., directed acyclic graphs), so that the function $F_i$ associated with the measurement of a node $i$ is of the form 
\begin{multline}\label{eq:function_Fi}
  \!\!\!\!\!  y_i^k=u_i^{k-1}+F_i(u_1^{k-1},\ldots,u_1^{k-M_1},\ldots,u_{n_i}^{k-1},\ldots,u_{n_i}^{k-M_{n_i}}),\\
    1,\dots,n_i\in \mathcal{N}^{e\to i},
\end{multline}
where $\mathcal{N}^{e\to i}$ is the set of excited nodes with a path to the \linebreak node $i$. The function $F_i$ determines the output of the node $i$ through the dynamics \eqref{eq:general_nonlinear} and the excitation signals of the nodes in $\mathcal{N}^{e\to i}$, and only depends on a finite number of inputs (i.e., $M_1,\ldots,M_{n_i}$ are finite) due to the finite delays $m_{i,j}$ and the absence of cycles. The model \eqref{eq:function_Fi}  is of the type Nonlinear Finite Impulse Response (NFIR) where the output is a nonlinear function only of the excitation signals (i.e., inputs) \cite{ramirez2021nonlinear,pillonetto2025deep}.
For the identifiability analysis, we consider the notion of identifiability in system identification \cite{ljung1999system}, where the objective is to determine if there exists a unique set of local dynamics $\Phi_i$ that leads to a global behavior given by the functions $F_i$. Hence, we make the following assumption for the derivation of the identifiability conditions.

\begin{assumption}\label{ass:function_Fi}
    If the node $i$ is measured, the function $F_i$ is known.
\end{assumption}

\begin{remark}[Functions $\Phi_i$ and $F_i$]
Although $\Phi_i$ and $F_i$ are functions associated with a node $i$, they are not the same since $\Phi_i$ determines the dynamics of $i$ while $F_i$ is the information obtained through the measurement of $i$. Notice that $\Phi_i$ is a function of the outputs of the in-neighbors and  is independent of the set of excited nodes $\NN^e$, while $F_i$ is a function of the excitation signals and depends on $\NN^e$.
\end{remark}

In this work, we will analyze the full measurement case where the set of measured nodes $\NN^m$ is given by $\NN^m=V$. This setting corresponds to networks where we have access to all the nodes, whose outputs can generally be measured using an appropriate sensing device. Our objective is to determine which nodes need to be excited to identify all the node functions $\Phi_i$. In this way, our aim is to determine the possibility of identification, and not to derive identification algorithms or study identification methods. This is different from the identifiability problem in structural equation models where the setting is probabilistic and the objective is the identifiability of probability distributions \cite{jordan2004graphical}. Fig.~\ref{fig:model} presents the framework considered in this work where all the nodes are measured, some nodes are excited and the node functions $\Phi_i$ depend on the outputs of the in-neighbors.

\subsection{Generic identifiability}\label{subsec:generic_ident}

We define the relationships between the measurements of the nodes and the functions $\Phi_i$.

\begin{definition}[Set of measured functions]\label{def:set_measured_functions}
    Given a set of excited nodes $\NN^e$, the totally ordered set of measured functions $(F(\mathcal{N}^e),\leq)$ associated with $\mathcal{N}^e$ is given by:
    $$
    F(\mathcal{N}^e):=\{F_i\;|\;i\in V\},
    $$
    with $F_i\leq F_j$ if $i\leq j$.
\end{definition}

For instance, let us consider again the DAG in Fig.~\ref{fig:model} with outputs of the nodes 2 and 3 of the form $y_2^k=\Phi_2(y_1^{k-1},y_3^{k-1})+u_2^{k-1}$ and $y_3^k=\Phi_3(y_1^{k-1})$ respectively. Since the nodes 1 and 2 are excited, the output of the node 2 can be expressed as $y_2^k=u_2^{k-1}+\Phi_2(u_1^{k-2},\Phi_3(u_1^{k-3}))$, where according to \eqref{eq:function_Fi}, the function $F_2$ is given by $F_2=\Phi_2(u_1^{k-2},\Phi_3(u_1^{k-3}))$. The set of measured functions is given by $F(\NN^e)=\{F_1,F_2,F_3,F_4,F_5 \}$ where each $F_i$ is a function of the excited nodes 1 and 2 similar to $F_2$.

Since the potential functions considered for the analysis of identifiability must in general satisfy some constraints 
depending on the type of network or application, we restrict the problem to certain classes of functions $\F_i$: a node function $\Phi_i$ belongs to $\F_i$ and the identifiability is considered only among the functions belonging to $\F_i$. 

We say that a digraph $G$ and a totally ordered set of node functions $(\{ \Phi\},\leq)=\{\Phi_i\in \F_i\;|\;i\in V\}$ with $\Phi_i\leq \Phi_j$ if $i\leq j$ generate $F(\mathcal{N}^e)$ if the functions $F_i \in F(\mathcal{N}^e)$ are recursively constructed from the node dynamics $\Phi_i$ via the dynamical equation \eqref{eq:general_nonlinear}. Since the ordering of the functions only depends on the labeling of the nodes, in the rest of the paper, we will refer to $F(\mathcal{N}^e)$ and  $\{ \Phi\}$ only as a set of measured functions and a set of node functions respectively.

While in the full excitation case \cite{vizuete2023nonlinear,vizuete2026nonlinear}, the set of measured functions depends on the set of measured nodes, in the full measurement case, the set of measured functions encompasses all the functions $F_i$, and its dependence is expressed with respect to the set of excited nodes.

Unlike \cite{vizuete2023nonlinear,vizuete2026nonlinear} where the notion of \emph{global identifiability}\footnote{Global identifiability refers to all functions as in \cite{vanwaarde2020necessary} and is different from the notion of global identifiability used to distinguish it from local identifiability as in \cite{legat2024identifiability}.} (i.e., for all functions) was used to derive the identifiability conditions, in this work, we will introduce the notion of \emph{generic identifiability} (i.e., for almost all functions) with respect to a vector $P$.

\begin{definition}[Generic Identifiability]\label{def:math_identifiability}
    Given a set of collection of parameterized node functions $\{ \Phi(P) \}=\{\Phi_i(P)\in \F_i\;|\;i\in V\}$ where each $P$ generates $F_P(\NN^e)$. A family of node functions $\Phi_i(P)$ is generically identifiable in the class $\F_i$ if $F_P(\NN^e)=F_{\tilde P} (\NN^e)$, implies that $\Phi_i(P)=\Phi_i(\tilde P)$ for any $P$ except possibly on a zero measure set. A network $G$ is generically identifiable in the class $\F=\F_1\times\F_2\times\cdots\times \F_n$ if $F_P(\NN^e)=F_{\tilde P} (\NN^e)$, implies that $\{\Phi(P) \}=\{ \Phi(\tilde P) \}$ for any $P$ except possibly on a zero measure set.
\end{definition}
The parametrization of the node functions  $\Phi_i(P)$ and the zero measure set of functions associated with the definition of generic identifiability will be clarified in Section~\ref{sec:generic}. Definition~\ref{def:math_identifiability} implies that a network $G$ is generically identifiable in the class $\F$ if all the node functions $\Phi_i$ are generically identifiable in the classes $\F_i$. Notice that Definition~\ref{def:math_identifiability} also covers the trivial case with no excitation signals (i.e., $\NN^e=\emptyset$). However, in this case all the functions $F_i$ would be identically zero and all the node functions $\Phi_i$ would be unidentifiable. Since each parameter $\tilde P$ denotes a specific set of node functions, in the rest of the work, we will use the notation $\tilde \Phi_i$, $\tilde F(\NN^e)$ and $\{ \tilde \Phi \}$ for $\Phi_i(\tilde P)$, $F_{\tilde P}(\NN^e)$ and $\{ \Phi(\tilde P) \}$ respectively.

In this work, we will consider analytic entire functions (i.e., for all $\hat x\in\R$, their Taylor series around $\hat x$ is convergent in a neighborhood of $\hat x$), whose properties will allow us to introduce finite-dimensional parametrization in the notion of generic identifiability.

\begin{definition}[Class of functions $\Fone_i$]\label{def:class_Ftwo}
Let $\Fone_i$ be the class of  functions $\Phi:\R^{m_i}\to\R$ where each function $\Phi$ is analytic.
\end{definition}

For each node function $\Phi_i$, the class of analytic functions will be determined by the number of arguments (i.e., delays) of $\Phi_i$ in \eqref{eq:general_nonlinear}. For instance, a node function $\Phi_1(y_2^{k-1},y_2^{k-2})$ belongs to the class of analytic functions $\Phi:\R^2\to\R$, while another node function $\Phi_2(y_3^{k-1},y_3^{k-2},y_3^{k-3})$ belongs to the class of analytic functions $\Phi:\R^3\to\R$. 
Therefore, in the rest of this work we will assume that $\Phi_i\in\Fone_i$ for all $i \in V$, and all the possible alternative functions $\tilde \Phi_i$ that could generate $F(\NN^e)$ must also belong to $\F^A_i$ (e.g., discontinuous functions are not considered even if they also generate $F(\NN^e)$).
For convenience, we will separate the univariate part from the rest of the function, so that a node function $\Phi_i$ can be rewritten as
\begin{multline}\label{eq:decomposition_Fi}
    \Phi_i=\sum_{j=1}^{M_i}\sum_{\ell_j=0}^{m_{i,j}} f_{i,j}^{\{\ell_j\}}(y_j^{k-\ell_j})+\\
    \!g_i(y_1^{k-0},\ldots,y_1^{k-m_{i,1}},\ldots,y_{M_i}^{k-0},\ldots,y_{M_i}^{k-m_{i,M_i}}),
\end{multline}
where the function $g_i$ encompasses all the terms corresponding to the product of two or more variables $y_j^{k-\ell}$, and a possible constant term, but not the terms depending only on one variable $y_j^{k-\ell}$. 
Hence, the identifiability of the functions $g_i$ and $f_{i,j}^{\{\ell\}}$ for $j\in\NN_i$ guarantees the identifiability of $\Phi_i$. By contrast, the additive model is of the form \cite{vizuete2023nonlinear,vizuete2026nonlinear}
\begin{equation}\label{eq:additive_model}
    y_i^k=u_i^{k-1}+\sum_{j\in \mathcal{N}_i}f_{i,j}(y_j^{k-0},\ldots,y_j^{k-{m_{i,j}}}),   
\end{equation}
\hspace{-3mm} where $f_{i,j}$ is a nonlinear function associated with the edge $(i,j)$, and the objective of the identifiability in this additive model is to determine conditions to identify the functions $f_{i,j}$ according to \cite[Definition 2]{vizuete2026nonlinear}. Moreover, 
by considering that each function $f_{i,j}$ is also additively separable in \eqref{eq:additive_model}, we obtain the \emph{fully additive} model \cite{vizuete2026nonlinear}: 
\begin{equation}\label{eq:fully_additive_model}
    y_i^k=u_i^{k-1}+\sum_{j=1}^{M_i}\sum_{\ell_j=0}^{m_{i,j}} f_{i,j}^{\{\ell_j\}}(y_j^{k-\ell_j}).   
\end{equation}

If for some functions $\Phi_i$ we have $g_i=0$ for all $i\in V$, then the network is fully additive. Also, the network becomes fully additive if the additional constraint $g_i=0$ for all $i\in V$ is imposed.
Unlike the additive model \eqref{eq:additive_model}, a static component does not affect the identifiability since we are interested on the identifiability of the node function $\Phi_i$ and not of the particular decomposition into $f_{i,j}$ and $g_i$, such that a static component can be included on any of the functions $f_{i,j}$ or $g_i$ without changing the function $\Phi_i$.  
In our setting, based on the particular decomposition \eqref{eq:decomposition_Fi}, it is assumed that the constant term is included in the function $g_i$.

\section{Generic nonlinear matrix for $g_i=0$}\label{sec:generic}

The notion of generic identifiability, according to which a network might not be identifiable only for a few particular cases of functions, can be motivated by \cite[Example~1]{vizuete2024partial}. In this section, we will use the decomposition \eqref{eq:decomposition_Fi} and consider the simpler problem where all the functions $g_i=0$ (i.e., the fully additive model). This simplification will be used as a mathematical tool to analyze the more general case of non-additive nonlinear networks.

We will formalize this notion of generic identifiability with respect to the functions $f_{i,j}^{\{\ell\}}$. First, we define the \emph{$K$-analytic parametrization consistent with a given digraph $G$ and delays} in the following way. Given a $K\in\mathbb{Z}^+$, for every function $f_{i,j}^{\{\ell\}}$ associated with an edge $(i,j)\in E$, set variables $\beta_{i,j,\ell}^{(p)}\in\R$, with $p\ge K+1$, and parametrize $f_{i,j}^{\{\ell\}}(y_j^{k-\ell})$ by
\begin{equation}\label{eq:parameterization}
f_{i,j}^{\{\ell\}}(y_j^{k-\ell})=\sum_{s=1}^K \alpha_{i,j,\ell}^{(s)}(y_j^{k-\ell})^{s}+\sum_{s=K+1}^\infty \beta_{i,j,\ell}^{(s)}(y_j^{k-\ell})^s,    
\end{equation}
for real parameters $\alpha_{i,j,\ell}^{(q)}$, with $1\le q\le K$, where we assume that the series $\sum_{s=K+1}^\infty \beta_{i,j,\ell}^{(s)}(y_j^{k-\ell})^{s}$ is convergent. For pairs $(i,j)$ not connected by an edge, let $f_{i,j}^{\{\ell\}}\equiv0$. We collect all parameters $\alpha_{i,j,\ell}^{(q)}$ in a vector $P$, and introduce the notion of a $K$-generic property related to the digraph $G$, which is related to the notion of genericity used in \cite{vanelli2025local}.
\begin{definition}[$K$-generic property]\label{def:K_generic_prop}
    We say that a property $K$-\emph{generically} holds for a given digraph $G$ if, for the $K$-analytic parametrization  consistent with the digraph $G$ and delays, the property holds for all parameters $\alpha_{i,j,\ell}^{(s)}$ except possibly those lying on a zero measure set, and for every $\beta_{i,j,\ell}^{(s)}$.
\end{definition}
Notice that in Definition~\ref{def:K_generic_prop}, the variables $\beta_{i,j,\ell}^{(s)}$ are not involved in the notion of a zero measure set, since the property depends on a finite-order truncation of the edge functions considering only the coefficients $\alpha_{i,j,\ell}^{(s)}$. Then, based on the notion a $K$-generic property, we define a generic property. 

\begin{definition}[Generic property]\label{def:generic_prop}
    We say that a property \emph{generically} holds for a given digraph $G$ if there exists a $K\in\mathbb{Z}^+$, such that the property is $K$-generic for any configuration of delays.
\end{definition}

Since in \eqref{eq:decomposition_Fi}, each function $f_{i,j}$ can be decomposed in the sum of functions of the form $f_{i,j}^{\{\ell\}}$, we can associate an edge to each function $f_{i,j}^{\{\ell\}}$ with an appropriate copy of the nodes $i$ and $j$. For this reason, we will use the \emph{unfolded digraph} at a node $v$ and time instant $k$, denoted by $H_v^k(G)$, which is constructed according to Algorithm~1 in \cite{vizuete2026nonlinear} by considering the initial time instant of the unfolded digraph as $t=0$. Since we are working with DAGs, the time instant $k$ can be considered sufficiently large and its dependence can be removed from the notation, so that we will only use $H_v(G)$. The unfolded digraph $H_v(G)=(V_H,E_H)$ is generated by creating copies of nodes at several time instants, so that the subscript in $i_m$ denotes the time instant of creation of the copy of the node $i\in V$, where $v_0$ is used to denote the origin of $H_v(G)$ that corresponds to the only sink. Each node has its own excitation signal and two nodes $j_t$ and $i_m$ are connected by a nonlinear function without delays $\pi_{i_m,j_t}(y_j^k)=f_{i,j}^{\{t-m\}}(y_j^k)$ if the node $j_t$ at time $t$ influences directly the node $i_m$ at time $m$. Unlike the full excitation case, if a node $i$ is not excited, we just consider a zero value for the excitation signals of all the copies $i_m$. For instance, let us consider a path graph with 3 nodes where the node 1 is excited and the nodes 2 and 3 are measured, and node functions of the form $\Phi_3(y_2^{k-1},y_2^{k-2})$ and $\Phi_2(y_1^{k-1},y_1^{k-2})$. Fig.~\ref{fig:unfolded_digraph} presents the unfolded digraph $H_3(G)$ associated with the path graph where the excitation signals of the copies of the nodes 2 and 3 are set to zero since the nodes are not excited.

\begin{figure}[!ht]
    \centering
    \begin{tikzpicture}
    [
roundnodes/.style={circle, draw=black!60, fill=black!5, very thick, minimum size=1mm},roundnode/.style={circle, draw=white!60, fill=white!5, very thick, minimum size=1mm},roundnodes2/.style={circle, draw=black!60, fill=black!25, very thick, minimum size=1mm},roundnodes3/.style={circle, draw=black!60, fill=white!30, very thick, minimum size=1mm}
]

\small

\node[roundnodes3](node4){$1_2$}(0,0.3) node[above] {$u_1^{k-2}$};
\node[roundnodes3](node5)[below=of node4,yshift=-0cm,xshift=0cm]{$1_3$}(0,-1.42) node[above] {$u_1^{k-3}$};
\node[roundnodes3](node6)[below=of node5,yshift=-0cm,xshift=0cm]{$1_4$}(0,-3.15) node[above] {$u_1^{k-4}$};
\node[roundnodes2](node8)[right=of node4,yshift=-8mm,xshift=2.25cm]{$2_1$}(4,-0.5) node[above] {$\cancelto{0}{u_2^{k-1}}$};
\node[roundnodes2](node9)[below=of node8,yshift=-1mm,xshift=0cm]{$2_2$}(4,-2.32) node[above] {$\cancelto{0}{u_2^{k-2}}$};
\node[roundnodes2](node11)[right=of node9,yshift=10mm,xshift=2.25cm]{$3_0$}(7.9,-1.32) node[above] {$\cancelto{0}{u_3^{k}}$};

\small

\draw[-{Classical TikZ Rightarrow[length=1.5mm]}] (node4) to node [above,swap,yshift=0mm,xshift=4mm] {$\pi_{2_1,1_2}\!=\! f_{2,1}^{\{1\}}$} (node8);
\draw[-{Classical TikZ Rightarrow[length=1.5mm]}] (node5) to node [above,swap,yshift=0mm,xshift=-4mm] {$\pi_{2_1,1_3}\!=\!f_{2,1}^{\{2\}}$} (node8);
\draw[-{Classical TikZ Rightarrow[length=1.5mm]}] (node5) to node [above,swap,yshift=0mm,xshift=4mm] {$\pi_{2_2,1_3}\!=\!f_{2,1}^{\{1\}}$} (node9);
\draw[-{Classical TikZ Rightarrow[length=1.5mm]}] (node6) to node [above,swap,yshift=0mm,xshift=-4mm] {$\pi_{2_2,1_4}\!=\!f_{2,1}^{\{2\}}$} (node9);
\draw[-{Classical TikZ Rightarrow[length=1.5mm]}] (node8) to node [above,swap,yshift=1mm,xshift=2mm] {$\pi_{3_0,2_1}\!=\!f_{3,2}^{\{1\}}$} (node11);
\draw[-{Classical TikZ Rightarrow[length=1.5mm]}] (node9) to node [below,swap,yshift=-0.5mm,xshift=2mm] {$\pi_{3_0,2_2}\!=\!f_{3,2}^{\{2\}}$} (node11);

\normalsize

\end{tikzpicture}
    \caption{Unfolded digraph $H_3(G)$ of a path graph $G$ with 3 nodes where the node 1 is excited and the nodes 2 and 3 are measured. The node functions are of the form $\Phi_3(y_2^{k-1},y_2^{k-2})$ and $\Phi_2(y_1^{k-1},y_1^{k-2})$. Several edge functions $\pi_{i_m,j_t}$ in $H_3(G)$ are given by the same edge functions $f_{i,j}^{\{\ell\}}$ of $G$.}
    \label{fig:unfolded_digraph}
\end{figure}

Now, we introduce some matrices that will be used in the proof of Theorem~\ref{thm:sufficiency}. We begin by defining a network matrix $J_\pi^0(\y^v)$ associated with the unfolded digraph $H_v(G)$ composed of the derivatives of the functions $\pi_{i_m,j_t}$:
\begin{equation}\label{eq:matrix_J_0}
J_\pi^0(\y^v)=
[\pi'_{i_m,j_t}(\y^v)],
\end{equation}
where $\y^v\in\R^{N_y}$ is a vector encompassing all the outputs of the nodes of $H_v(G)$ and can be considered as a free variable. This matrix can be interpreted as a nonlinear analogue of a weighted adjacency matrix where the entries are given by derivatives of the nonlinear functions $\pi'_{i_m,j_t}$ as a Jacobian type matrix.
Since the graph is acyclic, by applying \eqref{eq:fully_additive_model} to $J_\pi^0(\y^v)$, each entry of $J_\pi^0(\y^v)$ can be expressed as a function of $\uu^v\in \R^{N_u}$, which is a vector encompassing all the excitation signals in $H_v(G)$. Then, we define the \emph{nonlinear network matrix} $J_{H_v(G)}(\uu^v)$ as the matrix $J_\pi^0(\y^v)$ evaluated in $\y^v$ according to \eqref{eq:fully_additive_model}:
\begin{equation}\label{eq:matrix_J_G}
J_{H_v(G)}(\uu^v)=[\Pi_{i_m,j_t}(\uu^v)],    
\end{equation}
where each $\Pi_{i_m,j_t}$ is an analytic function since it is obtained by sums and compositions of the functions $f_{i,j}^{\{\ell\}}$. While $J_\pi^0(\y^v)$ describes a behavior of the network with respect to the outputs of nodes, the matrix $J_{H_v(G)}(\uu^v)$ describes the behavior of the network with respect to the excitation signals, which are essential in the identifiability in the full measurement case.
Notice that the coefficients of the Taylor series of each $\Pi_{i_m,j_t}$ is a function of the parameters  $\alpha_{i,j,\ell}^{(s)}$  and $\beta_{i,j,\ell}^{(s)}$ in \eqref{eq:parameterization}.

In our work, we are particularly interested in the genericity of the rank of the square submatrices of $J_{H_v(G)}(\uu^v)$, which will be used in the proof of Theorem~\ref{thm:sufficiency}. The maximal rank of a square submatrix $J_{H_v(G)}^{A,B}(\uu^v)$ is the highest possible rank. Given a square submatrix $J_{H_v(G)}^{A,B}(\uu^v)$, we say that its rank is \emph{generic} if having the maximal rank for almost all $\uu^v$ holds generically over the $K$-parametrizations. Notice that the genericity of the rank is considered only with respect to the parameters $\alpha_{i,j,\ell}^{(s)}$, and excludes the excitation signals $\uu^v$.

\begin{proposition}\label{prop:generic_rank_submatrix}
    The rank of any square submatrix $J_{H_v(G)}^{A,B}(\uu^v)$ is generic for almost all $\uu^v$.
\end{proposition}
\begin{proof}
    The determinant of any square submatrix $J_{H_v(G)}^{A,B}(\uu^v)$, denoted by $\detm(J_{H_v(G)}^{A,B}(\uu^v))$ is given by sums and products of analytic functions of the form $\pi'_{i_m,j_t}$, so that $\detm(J_{H_v(G)}^{A,B}(\uu^v))$ is also analytic and can be expressed as
\begin{multline}\label{eq:determinant_terms}
\detm(J_{H_v(G)}^{A,B}(\uu^v))=\\
\sum_{p_1,\ldots,p_{N_u}=0} a_{p_1,\ldots,p_{N_u}}(\uu^v_1)^{p_1}\cdots (\uu^v_{N_u})^{p_{N_u}},
\end{multline}
If the determinant is not identically zero, there must exist at least one coefficient $a_{\hat p_1,\ldots,\hat p_{N_u}}\neq 0$. Let us denote by $a_{\min}$ one of the nonzero coefficients corresponding to the minimum degree $D_{\min}$ of the terms in \eqref{eq:determinant_terms}. This coefficient is given  by a polynomial $p(P_{\min})$, where $P_{\min}\in\R^K$ is a vector that encompasses the coefficients $\alpha_{i,j,\ell}^{(s)}$ and variables $\beta_{i,j,\ell}^{(s)}$ of the functions $f_{i,j}^{\{\ell\}}$ in \eqref{eq:parameterization} associated with terms whose degree is at most $D_{\min}$, and $K$ is finite and nonzero. 
Notice that this polynomial   $p(P_{\min})$ can be zero only on a subspace of dimension at most $K-1$. Then, the rank of the submatrix $J_{H_v(G)}^{A,B}$ is maximal for all $P_{\min}\in\R^{K}$, except possibly on a subspace of dimension at most $K-1$, which has measure zero. Since one of the coefficients in \eqref{eq:determinant_terms} is nonzero, the determinant is nonzero for almost all $\uu^v$, given that the determinant is an analytic function of $\uu^v$ and can only be zero in a zero measure set in the space of excitation signals. Therefore, the rank of any square submatrix $J_{H_v(G)}^{A,B}$ is $K$-generic and since it holds for this particular value of $K$, it is also generic. 
\end{proof}
The coefficients in \eqref{eq:determinant_terms} can be considered as elements of the space of infinite sequences with the infinite norm $(\ell^\infty,\norm{\cdot}_\infty)$. In this case, the set of coefficients that satisfies the generic property is open and dense in $(\ell^\infty,\norm{\cdot}_\infty)$, which implies that it is also Baire-generic \cite{oxtoby1980measure}.

\section{Directed Acyclic Graphs}\label{sec:DAGs}

\subsection{Sufficient condition}

First, we will analyze the role of the sources and sinks in the identifiability of DAGs in the  model \eqref{eq:general_nonlinear}.

\begin{lemma}\label{lemma:sinks_sources}
    The outgoing edges of sources are not identifiable if the sources are not excited. The excitation of sinks is never necessary for the identifiability of the network. The measurement of sources is never necessary.
\end{lemma}
\begin{proof}
Without loss of generality, let us consider that the node 1 is a source and $j$ is an out-neighbor of 1.
    The measurement of $j$ provides the output
    \begin{equation}\label{eq:excitation_source1}
        y_j^k=u_j^{k-1}+\Phi_j(y_1^{k-0},\ldots,y_1^{k-m_{j,1}},y_2^{k-0},\ldots,y_{M_j}^{k-m_{j,M_j}}).    
    \end{equation}
    If the source $1$ is not excited (i.e., $u_1^k=0$ for all $k$), its output $y_1^k=0$, and   \eqref{eq:excitation_source1} becomes
    \begin{equation}\label{eq:excitation_source2}
        y_j^k=u_j^{k-1}+\Phi_j(0,\ldots,0,y_2^{k-0},\ldots,y_{M_j}^{k-m_{j,M_j}}).  
    \end{equation}
    Notice that any node function $\tilde \Phi_j=\psi_1(y_1^{k-0},\ldots,y_1^{k-m_{j,1}})+\Phi_j$ where $\psi_1$ is analytic and  $\psi_1(0)=0$, also satisfies \eqref{eq:excitation_source2}. This implies that it is not possible to identify $\Phi_j$ and hence, the excitation of all the sources is necessary for identifiability of the network.
    Now, the excitation of a sink $\ell$ can only affect its output:
    $y_\ell^k=u_\ell^{k-1}+\Phi_\ell$. Therefore,
    the identifiability of $\Phi_\ell$ is independent of the value of $u_\ell^{k-1}$, which implies that the excitation of the sink is never necessary for identifiability.
Finally, the measurement of a source $i$ provides the output $y_i=u_i^{k-1}$ that does not include any unknown dynamics.
\end{proof}
Next, we establish a link between vertex-disjoint paths and the rank of submatrices of $T_{H_v(G)}(\uu^v)=(I-J_{H_v(G)}(\uu^v))^{-1}$, which is well defined since $J_{H_v(G)}(\uu^v)$ is upper-triangular in a DAG \cite{hendrickx2019identifiability}.
\begin{definition}[Vertex-Disjoint Paths \cite{hendrickx2019identifiability}]
    A group of paths are mutually vertex disjoint if no two paths of this group contain the same vertex.
\end{definition}

\begin{proposition}\label{prop:generic_rank_vertex_disjoint}
Given a DAG $G$ and a node $i$ where there are vertex-disjoint paths from excited nodes to the in-neighbors of $i$. Let us denote by $A$ a subset of the copies of the in-neighbors of $i$ in $H_i(G)$. Then, there exists a subset of excited nodes $B$ in $H_i(G)$ with $|B|=|A|$ such that the generic rank of the submatrix $T_{H_i(G)}^{A,B}(\uu_G^k)$ is full.   
\end{proposition}
\begin{proof}
First, we will prove that there are vertex-disjoint paths from excited nodes to $A$ in $H_i(G)$.  Let us select an arbitrary set of vertex-disjoint paths in $G$, and let us  consider an arbitrary in-neighbor of $i$, denoted by $j$, which is reached through a path from the excited node $q$ in $G$. Notice that by the choice of the set of vertex-disjoint paths, the excited node $q$ can only reach $j$ in $G$. Now, let us consider a copy of $j$ with the smallest delay $j_m$, which is reached by a copy of $q$ denoted by $q_t$ through a path $P_{q_t\to j_m}$ in $H_i(G)$. A  copy of $j$ with a different delay $j_{m+\delta}$ must necessarily be reached by a different copy $q_{t+\delta}$ through a different path $P_{q_{t+\delta}\to j_{m+\delta}}$ constructed by copies of the nodes in the path $P_{q_t\to j_m}$ delayed by $\delta$. Therefore, the paths $P_{q_t\to j_m}$ and $P_{q_{t+\delta}\to j_{m+\delta}}$ are vertex-disjoint. By applying the same procedure to other copies of $j$ in $H_i(G)$ and other in-neighbors of $i$ in $G$, we guarantee the existence of a set of vertex-disjoint paths from excited nodes to $A$.

\noindent    Now, we will prove the existence of the subset of excited nodes $B$.
    From Proposition~\ref{prop:generic_rank_submatrix}, the rank of
    any square submatrix $T_{H_i(G)}^{A,B}(\uu^i)$ is generic. Let us consider
    the particular case of linear functions of the form $f_{i,j}^{\{\ell\}}(y_j^{k-\ell})=\varphi_{i,j,\ell}y_j^{k-\ell}$, which clearly belong to the class $\Fone$. In the graph $G$, let us select a set of vertex-disjoint paths from excited nodes to the in-neighbors of $i$, and for each function $f_{i,j}$ in one of these paths, we set $\varphi_{i,j,s_{i,j}}=1$ where $s_{i,j}$ is the minimum delay. The rest of parameters $\varphi_{i,j,\ell}$ are set to zero. 
    For this type of functions, the entries of the matrix $J_\pi^0(\y^i)$ in \eqref{eq:matrix_J_0} are constant
    values given by $\varphi_{i,j,\ell}$ and since they 
    are independent of $\y^i$, the matrix $J_{H_i(G)}(\uu^i)$ 
    in \eqref{eq:matrix_J_G} satisfies $J_{H_i(G)}(\uu^i)=J_\pi^0(\y^i)$.  
    According to the proof of Proposition~V.1 in \cite{hendrickx2019identifiability}, for $a_m\in A$ and $b_t\in B$ an entry $[T_{H_i(G)}]_{b_t,a_m}=1$ only if $a_m$ and $b_t$ are on the same path. Let us consider an in-neighbor $j$ of $i$ that is reached from the excited node $q$.
    A copy of $j$ with delay $m$ denoted by $j_m$, is reached only by the copy of the node $q$ with delay $t$, denoted by $q_t$, where $t=m+\sum_{(a,b)\in P_{q\to j}}s_{a,b}$. Other copies of $j$ must be reached necessarily from other copies of $q$ because of the different delays. 
    Therefore, the matrix $T_{H_i(G)}^{A,B}(\uu^i)$ is a permutation matrix and has full rank. Finally,
    since the rank is generic and we showed that it is full for this particular case, the proof is completed.
\end{proof}

Now, we provide a sufficient condition based on vertex-disjoint paths for the generic identifiability of DAGs.

\begin{theorem}\label{thm:sufficiency}
    In the full measurement case, a DAG is generically identifiable in the class $\Fone$ if there are vertex-disjoint paths from excited nodes to the in-neighbors of each node.
\end{theorem}

Before presenting the proof of Theorem~\ref{thm:sufficiency}, we recall a
technical result that will be used in the proof.

\begin{lemma}[Theorem 2.35 \cite{laczkovich2017real}]\label{lemma:mapping}
    Let $H\subset \R^p$ and let $f:H\to\R^q$, where $p\ge q$. If $f$ is continuously differentiable at the point $a\in \text{int} \;H$ and the linear mapping $f'(a):\R^p\to\R^q$ is surjective, then the range of $f$ contains a neighborhood of $f(a)$.
\end{lemma}
\begin{myproof}{Theorem~\ref{thm:sufficiency}}
We proceed by induction. For a DAG and a topological ordering, we take a node $i$ and we denote without loss of generality by $\NN_i=\{1,\ldots,p\}$ the set of in-neighbors. We assume by induction that all the dynamics $F_j$ preceding node $i$ in the topological ordering have been identified. In the basic case of $\NN_i$ being empty ($i$ is a source), this assumption is trivial.
The measurement of $i$ is given by:
\begin{align*}
    y_i^k&=u_i^{k-1}+\Phi_i(y_1^{k-0},\ldots,y_1^{k-m_{i,1}},\ldots,y_{M_i}^{k-0},\ldots,y_{M_i}^{k-m_{i,M_i}})\\
    &=\sum_{j=1}^p\sum_{\ell_j=0}^{m_{i,j}} f_{i,j}^{\{\ell_j\}}(\xi_{i,j}^{\{\ell_j\}}(\uu_e^i))+g_i(\xi_{i,1}^{\{0\}},\ldots,\xi_{i,p}^{\{m_{i,p}\}})\\
    &=F_i(\xi_{i,1}^{\{0\}},\ldots,\xi_{i,p}^{\{m_{i,p}\}}),
\end{align*} 
where $\uu_e^i\in \R^{N_e^i}$ is a vector with all the excitation signals that arrive to $i$ and $\xi_{i,j}^{\{\ell_j\}}=u_j^{k-\ell_j-1}\!+\!F_j(\uu_e^i)$ is the output of the in-neighbor $j$ of $i$ with its corresponding delay $\ell_j$ as a function of all the excitation signals that arrive to $i$.
    Let us assume that there exists a set $\{ \Phi \}\neq \{ \tilde \Phi \}$ such that $F(\NN^e)=\tilde F(\NN^e)$. Since $i\in \NN^m$, the measured functions $F_i$ and $\tilde F_i$ must satisfy:
\begin{equation}\label{eq:identifiability_problem}
    F_i(\xi_{i,1}^{\{0\}},\ldots,\xi_{i,p}^{\{m_{i,p}\}})=\tilde F_i(\xi_{i,1}^{\{0\}},\ldots,\xi_{i,p}^{\{m_{i,p}\}}),
\end{equation}
given that all the dynamics preceding node $i$ in the topological ordering are known (i.e., $\xi_{i,j}^{\{\ell\}}=\tilde \xi_{i,j}^{\{\ell\}}$ for all $j\in \NN_i$ and all $\ell$) according to the induction.
Now, let us define the mapping $\Xi_i:\R^{N_e^i}\to\R^{N_p^i}$ 
$$
\Xi_i(\uu_e^i)\!=\!(\xi_{i,1}^{\{0\}}(\uu_e^i),\ldots,\xi_{i,p}^{\{m_{i,p}\}}(\uu_e^i)),
$$
which sends the $N_e^i$ excitation signals with a path to the node $i$ to the $N_p^i$ outputs of the in-neighbors of $i$ considering different delays. We denote the Jacobian matrix of $\Xi_i$ as $J_{\Xi_i}(\uu_e^i)$. Now, let us consider the subset of in-neighbors of $i_0$ in the unfolded digraph $H_i(G)$ denoted by $A$.
According to Proposition~\ref{prop:generic_rank_vertex_disjoint}, there exists a set of excitation signals $B$ such that the generic rank of $T_{H_i(G)}^{A,B}(\uu^i)$ is full rank. This implies that the determinant $\detm(T_{H_i(G)}^{A,B}(\uu^i))$ is not zero and there exists at least a minimum non-zero coefficient $a_{\min}(P_{\min})$ where the vector of parameters $P_{\min}\in\R^K$ is associated with coefficients of the functions $f_{i,j}^{\{\ell\}}$ according to \eqref{eq:parameterization}. 

\noindent Now, let us analyze the rank of the Jacobian matrix restricted to the set of excitation signals $B$ denoted by $J_{\Xi_i}^B (\uu_e^i)$. The determinant  $\detm(J_{\Xi_i}^B (\uu_e^i))$ is an analytic function where each coefficient is a function of a finite number of coefficients of the functions $f_{i,j}^{\{\ell\}}$ and $g_i$ depending on the degree of the coefficient of $\detm(J_{\Xi_i}^B (\uu_e^i))$. Let us consider a coefficient $b_{\min}(Q_{\min})$ where $Q_{\min}=(P_{\min} \quad G)^T\in\R^L$ and $G$ is a vector with a finite number of coefficients of the functions $g_i$. Then, let us consider the particular case $G=0$, where  $b_{\min}=a_{\min}$ is non zero, and it could only be zero in a subspace of dimension at most $L-1$, which has zero measure.

\noindent Since the determinant $\detm(J_{\Xi_i}^B (\uu_e^i))$ is an analytic function of $\uu_e^i$, it can be zero everywhere or only on a measure-zero set \cite{krantz2002primer}. This implies that there exists at least a point $\hat \uu_e^i\neq 0$ such that $\detm(J_{\Xi_i}^A (\hat \uu_e^i))\neq 0$ and hence, the generic rank of  $J_{\Xi_i}( \uu_e^i)$ is full for almost all $\uu_e^i$. Therefore, 
$J_{\Xi_i}(\hat \uu_e^i)$ is surjective and
by virtue of Lemma~\ref{lemma:mapping}, the range of $\Xi_i$ must contain a neighborhood of $\Xi_i(\hat \uu_e^i)$, which implies that \eqref{eq:identifiability_problem} holds on a set of positive measure. Then, by the Identity Theorem of analytic functions \cite{krantz2002primer}, we guarantee that $F_i=\tilde F_i$ everywhere, which implies that $\Phi_i=\tilde \Phi_i$ everywhere.
Hence, the node function $\Phi_i$ is generically identifiable with respect to the parameters  encompassed in the vector $Q_{\min}$.

\noindent Now, notice that for $i=2$ in the topological ordering, if there is an edge $f_{2,1}$, the functions $\xi_{2,1}^{\{\ell\}}$ are the identity function and the node function $\Phi_2$ can be clearly identified. Then, by induction, the identifiability analysis is valid for any node in the DAG. Thus, the entire DAG is generically identifiable.
\end{myproof}

A direct consequence of Theorem~\ref{thm:sufficiency} and Lemma~\ref{lemma:sinks_sources} is a sufficient and necessary condition for  identifiability of trees.

\begin{proposition}\label{prop:trees}
    In the full measurement case, a tree is generically identifiable in the class $\Fone$ if and only if all the sources are excited.
\end{proposition}

\subsection{Necessity of the vertex-disjoint path condition}

\begin{figure}[!ht]
    \centering
    \begin{tikzpicture}
    [
roundnodes/.style={circle, draw=black!60, fill=black!5, very thick, minimum size=1mm},roundnode/.style={circle, draw=white!60, fill=white!5, very thick, minimum size=1mm},roundnodes2/.style={circle, draw=black!60, fill=black!25, very thick, minimum size=1mm},roundnodes3/.style={circle, draw=black!60, fill=white!30, very thick, minimum size=1mm}
]

\small

\node[roundnodes3](node1){1};
\node[roundnodes2](node4)[right=of node1,yshift=6mm,xshift=0cm]{2};
\node[roundnodes2](node5)[right=of node1,yshift=-6mm,xshift=0cm]{3};
\node[roundnodes2](node7)[right=of node4,yshift=-6mm,xshift=0cm]{4};

\draw[-{Classical TikZ Rightarrow[length=1.5mm]}] (node1) to node [above,swap,yshift=0mm] {} (node4);
\draw[-{Classical TikZ Rightarrow[length=1.5mm]}] (node1) to node [above,swap,yshift=0mm] {} (node5);
\draw[-{Classical TikZ Rightarrow[length=1.5mm]}] (node4) to node [above,swap,yshift=0mm] {} (node7);
\draw[-{Classical TikZ Rightarrow[length=1.5mm]}] (node5) to node [above,swap,yshift=0mm] {} (node7);

\normalsize

\end{tikzpicture}
    \caption{DAG that is identifiable for a particular choice of functions even if there are no vertex-disjoint paths from excited nodes to the in-neighbors of node 4.}
    \label{fig:example}
\end{figure}

The vertex-disjoint path condition of Theorem~\ref{thm:sufficiency} is not necessary for certain cases as we can see from the following counterexample. 

\begin{example}\label{ex:Dirichlet}
  Consider the DAG in Fig.~\ref{fig:example} with node functions $\Phi_2(y_1^{k-1})=e^{qy_1^{k-1}}$ with $q$ irrational, $\Phi_3(y_1^{k-1})=e^{y_1^{k-1}}$ and $\Phi_4(y_2^{k-1},y_3^{k-1})$. The measurement of the node 4 provides the output
$
y_4^k=\Phi_4(\Phi_2(u_1^{k-2}),\Phi_3(u_1^{k-2})).
$
Let us consider that there is another function $\tilde \Phi_4$ such that:
$$
\Phi_4(\Phi_2(u_1^{k-2}),\Phi_3(u_1^{k-2}))=\tilde \Phi_4(\Phi_2(u_1^{k-2}),\Phi_3(u_1^{k-2})),
$$
which is equivalent to
$
\Delta \Phi_4(\Phi_2(u_1^{k-2}),\Phi_3(u_1^{k-2}))=0.
$
Since $\Delta \Phi_4$ is analytic, its series is given by:
$$
\Delta \Phi_4 =\sum_{m,n=0}^\infty a_{m,n} e^{(qm+n)u_1^{k-2}},
$$
which is a Dirichlet series since $qm+n$ is irrational. This implies that every coefficient $a_{m,n}$ must be zero to satisfy $\Delta \Phi_4=0$. Therefore, $\Delta \Phi_4\equiv 0$ and $\Phi_4$ is identifiable despite not satisfying the vertex-disjoint path condition.
\end{example}

This shows that it is not true that the vertex-disjoint path condition determines if a DAG is identifiable for almost all functions or for none as in \cite{hendrickx2019identifiability}. However, it remains an open question to determine if in the absence of vertex-disjoint paths, we have identifiability for almost no function as in pseudo-genericity \cite{legat2024identifiability}.

The vertex-disjoint path condition plays an important role if we consider the class of polynomials functions and we extend each class $\F_i$ by considering more potential delays than in the function $\Phi_i$ in \eqref{eq:general_nonlinear}. For instance, for a node function $\Phi_1(y_2^{k-1},y_2^{k-2})$ we can consider a class of polynomial functions with more than 2 arguments.

\begin{definition}[Class of functions $\Fthree_i$]\label{def:class_Fthree} Let $\Fthree_i$ be the class of functions $\Phi:\R^{m_i}\to\R$ where each function $\Phi$ is polynomial.
\end{definition}

Next, we introduce the notion of an affine algebraic variety in the field $\R$.

\begin{definition}[Affine algebraic variety]
    Let $f_1,\ldots,f_s$ be polynomials in $\R[x_1,\ldots,x_n]$. Then we call $\V(f_1,\ldots,f_s)$ the affine algebraic variety defined as $\V(f_1,\ldots,f_s)= \{ (a_1,\ldots,a_n)\in \R^n\ | f_i(a_1,\ldots,a_n)=0 \text{ for all } 1\le i\le s\}.$
\end{definition}

The following lemma is a direct consequence of the Polynomial Implicitization Theorem \cite{cox2015ideals}.

\begin{lemma}\label{lemma:algebraic_variety}
Let $\V\subseteq \R^n$ be given parametrically as
    \begin{align*}
    x_1&=f_1(t_1,\ldots,t_m)\\
        &\;\,\vdots\\
        x_n&=f_n(t_1,\ldots,t_m),
    \end{align*}
    where $f_1,\ldots,f_n$ are polynomials in $\R[t_1,\ldots,t_m]$ and $n>m$. Then $\V$ is included in an affine algebraic variety $\mathcal{W}(F)$,  for a polynomial $F\in \R[x_1,\ldots,x_n]$ that is not identically zero. 
\end{lemma}

\begin{theorem}\label{thm:necessity}
    In the full measurement case, if in a DAG there are no vertex-disjoint paths from excited nodes to the in-neighbors of each node, then the DAG is unidentifiable.

\end{theorem}
\begin{proof}
    Let us consider an arbitrary node $i$ with $M_i$ in-neighbors, which yields the following identifiability problem
\begin{multline}\label{eq:necessity_1}
\Phi_i(y_1^{k-0},\ldots,y_1^{k-m_{i,1}},\ldots,y_{M_i}^{k-0},\ldots,y_{M_i}^{k-m_{i,M_i}})=\\
\tilde \Phi_i(y_1^{k-0},\ldots,y_1^{k-m_{i,1}},\ldots,y_{M_i}^{k-0},\ldots,y_{M_i}^{k-m_{i,M_i}},\ldots),    
\end{multline}
where $\tilde \Phi_i$ might be a function of more arguments (delayed variables) than the real function $\Phi_i$ since the class $\Fthree_i$ is extended to more possible delays. 
According to \eqref{eq:function_Fi}, the outputs of the nodes can be expressed as a function of the excitation signals corresponding to the nodes in $\NN^{e\to i}$, such that \eqref{eq:necessity_1} becomes
\begin{multline}\label{eq:equality_necessity}
    \Phi_i(F_1^{\{0\}}(\uu_e^i),\ldots,F_{M_i}^{\{m_{i,M_i}\}}(\uu_e^i))\\
    =\tilde \Phi_i(F_1^{\{0\}}(\uu_e^i),\ldots,F_{M_i}^{\{m_{i,M_i}\}}(\uu_e^i),\ldots),
\end{multline}
where $\uu_e^i$ encompasses all the excitation signals that arrive to $i$ with their corresponding delays, and we assume that all the dynamics $F_j$ of the in-neighbors of the node $i$ have been identified. Let us consider a potential function $\tilde\Phi_i$ of the form $\tilde\Phi_i=\Phi_i+H_i$. Notice that this function $\tilde\Phi_i$ also satisfies \eqref{eq:equality_necessity} if there exists a function $H_i\not\equiv0$ such that
\begin{equation}\label{eq:necessity_0}
H_i(F_1^{\{1\}}(\uu_e^i),\ldots,F_{M_i}^{\{m_{i,M_i}\}}(\uu_e^i),\ldots)=0.
\end{equation} 
Let us assume that any set of vertex-disjoint paths from excited nodes $\NN^{e\to i}$ can only reach at most $D_i$ in-neighbors of $i$, where $D_i<M_i$. 
From \cite[Lemma~V.3]{hendrickx2019identifiability}, the size of the smallest $\NN^{e\to i}-\NN_i$ disconnecting set\footnote{A set of nodes $\mathcal{B}$ is an $\mathcal{A}-\mathcal{C}$ disconnecting set if every path starting in $\mathcal{A}$ and ending in $\mathcal{C}$ contains at least one node in $\mathcal{B}$, which implies that if $\mathcal{B}$ is removed, there will be no path from $\mathcal{A}$ to $\mathcal{C}$ \cite{hendrickx2019identifiability}.}, denoted by $\mathcal{D}_{e\to i}$, is the maximum number of vertex disjoint paths given by $D_i$. Since each path from the set of excited nodes $\NN^{e\to i}$ must cross $\mathcal{D}_{e\to i}$, the output of each in-neighbor of $i$ can be expressed as a function of the outputs $\z_1,\ldots,\z_{D_i}$, where
$\z_j=(z_j^{k-0},\ldots,z_j^{k-T_j})$ includes all the outputs of the node $z_j\in \mathcal{D}_{e\to i}$ from $z_j^{k-0}$ until the largest delay $z_j^{k-T_j}$. Considering that the output of each node $z_j\in \mathcal{D}_{e\to i}$ can be expressed as a function of the excitation signals $\uu_e^i$, \eqref{eq:necessity_0} is given by
\begin{equation*}
    H_i(G_1^{\{0\}}(\z_1,\ldots,\z_{D_i}),\ldots,G_{M_i}^{\{m_{i,M_i}\}}(\z_1,\ldots,\z_{D_i}),\ldots)=0.    
\end{equation*}
Let us denote by $N_f$ the dimension of the domain of the function $\Phi_i$ and by $N_\z$ the total number of variables $z_j$ with the corresponding delays.
Similarly, we denote by $\tilde N_{f}$ the dimension of the domain of $H_i$ and by $\tilde N_\z$ the total number of variables $z_j$ in $H_i$. Notice that for each in-neighbor $j$ of $i$,  additional arguments from $G_j^{\{m_{i,j}\}}$ until $G_j^{\{m_{i,j+\delta}\}}$ imply $\tilde N_f=N_f+\delta M_i$,  but only $\tilde N_\z=N_\z+\delta D_i$, since only the last $\delta$ variables have new delays. Then, for  
$\delta>\frac{N_f-N_\z}{M_i-D_i}$, we can
apply Lemma~\ref{lemma:algebraic_variety} to guarantee that there exists a function $H_i \not\equiv 0$ such that $H_i(G_1^{\{0\}}(\z_1,\ldots,\z_{D_i}),\ldots,G_{M_i}^{\{m_{i,M_i}\}}(\z_1,\ldots,\z_{D_i}),\ldots)=0$ for all $\z_1,\ldots,\z_{D_i}$. 
Therefore, the node function $\Phi_i$ is not unique and not identifiable with the information obtained through the measurement of node $i$.   
\end{proof}

The proof of Theorem 2 shows that if the vertex-disjoint path condition is not satisfied, by considering more delays, we can always find another polynomial function that satisfies the information obtained with the measurement of the nodes.

\section{Non-additive and additive model}\label{sec:linear}

In \cite{vizuete2024partial}, a similar identifiability condition based on vertex-disjoint paths was obtained for the identifiability of DAGs in the additive model and the class of pure nonlinear functions.  
Unfortunately, the following counterexample shows that this sufficient condition is not necessary for identifiability in the class of pure nonlinear functions in the additive model \cite{vizuete2024partial}.

\begin{figure}[!ht]
    \centering
    \begin{tikzpicture}
    [
roundnodes/.style={circle, draw=black!60, fill=black!5, very thick, minimum size=1mm},roundnode/.style={circle, draw=white!60, fill=white!5, very thick, minimum size=1mm},roundnodes2/.style={circle, draw=black!60, fill=black!25, very thick, minimum size=1mm},roundnodes3/.style={circle, draw=black!60, fill=white!30, very thick, minimum size=1mm}
]

\small

\node[roundnodes3](node1){1};
\node[roundnodes3](node5)[below=of node1,yshift=-0.4cm,xshift=0cm]{2};
\node[roundnodes2](node2)[right=of node1,yshift=0mm,xshift=0.8cm]{3};
\node[roundnodes2](node3)[right=of node5,yshift=0mm,xshift=0.8cm]{5};
\node[roundnodes2](node6)[below=of node2,yshift=0.6cm,xshift=-1.5cm]{4};
\node[roundnodes2](node7)[right=of node6,yshift=0cm,xshift=-0.1cm]{6};
\node[roundnodes2](node8)[right=of node7,yshift=0cm,xshift=0.4cm]{7};

\draw[-{Classical TikZ Rightarrow[length=1.5mm]}] (node1) to node [above,swap,yshift=0mm] {} (node2);
\draw[-{Classical TikZ Rightarrow[length=1.5mm]}] (node5) to node [above,swap,yshift=0mm] {} (node3);
\draw[-{Classical TikZ Rightarrow[length=1.5mm]}] (node1) to node [above,swap,yshift=0mm] {} (node6);
\draw[-{Classical TikZ Rightarrow[length=1.5mm]}] (node5) to node [above,swap,yshift=0mm] {} (node6);
\draw[-{Classical TikZ Rightarrow[length=1.5mm]}] (node6) to node [above,swap,yshift=0mm] {} (node7);
\draw[-{Classical TikZ Rightarrow[length=1.5mm]}] (node7) to node [above,swap,yshift=0mm] {} (node8);
\draw[-{Classical TikZ Rightarrow[length=1.5mm]}] (node2) to node [above,swap,yshift=0mm] {} (node8);
\draw[-{Classical TikZ Rightarrow[length=1.5mm]}] (node3) to node [above,swap,yshift=0mm] {} (node8);

\normalsize

\end{tikzpicture}
    \caption{A DAG where the excitation of the sources 1 and 2 is sufficient for identifiability in the additive nonlinear model and the class of pure nonlinear functions. However, according to Theorem~\ref{thm:necessity}, the DAG is unidentifiable in the non-additive nonlinear model and the class $\Fthree$.
    }
    \label{fig:DAG_counterexample}
\end{figure}

\begin{example}\label{ex:counterexample}
   Consider the DAG in Fig.~\ref{fig:DAG_counterexample}. Since node 7 has 3 in-neighbors (3,5,6) and we only have two excited nodes (1,2), there are no vertex-disjoint paths from excitations to the in-neighbors of 7 and the DAG is unidentifiable in the non-additive model. For the additive model \eqref{eq:additive_model} and pure nonlinear functions, the measurement of the node 7 is of the form  
   \begin{multline*}
f_{7,3}(g_{7,3}(u_1))+f_{7,6}(g_{7,6}(u_1,u_2))+f_{7,5}(g_{7,5}(u_2))=\\
\tilde f_{7,3}(g_{7,3}(u_1))+\tilde f_{7,6}(g_{7,6}(u_1,u_2))+\tilde f_{7,5}(g_{7,5}(u_2)),    
\end{multline*}
which is equivalent to
\begin{equation*}
\Delta f_{7,3}(g_{7,3}(u_1))+\Delta f_{7,6}(g_{7,6}(u_1,u_2))+\Delta f_{7,5}(g_{7,5}(u_2))=0,    
\end{equation*}
where we omit the time dependence of the excitation signals.
The function $g_{7,6}$ must contain terms of the form $u_1^pu_2^q$, and if $\Delta f_{7,6}\not\equiv 0$, these terms  cannot be canceled by $\Delta f_{7,3}$ and $\Delta f_{7,5}$ that only depend on $u_1$ and $u_2$ respectively. Similarly, the function $\Delta f_{7,3}$ or $\Delta f_{7,5}$ cannot be different from zero. Therefore, all the functions $\Delta f_{7,3}$, $\Delta f_{7,6}$ and $\Delta f_{7,5}$ must be necessarily 0, and  the DAG is identifiable in the additive model despite not satisfying the vertex-disjoint path condition.
\end{example}

This distinction between the additive and the non-additive nonlinear models arises from the separability of the node function associated to the additive model, which limits considerably the potential class of functions that can satisfy the information obtained with the measurement of a node. However, notice that for the linear case, where the functions are necessarily separable, the vertex-disjoint path condition is also necessary. This can be explained by the nonlinearity of the functions, which might generate terms involving the product of two or more excitation signals, allowing us to identify a node function with a reduced number of excitation signals.

\section{Conclusions and future work}\label{sec:conclusions}

In this paper, we analyzed the identifiability of a network with a non-additive nonlinear model that can be expressed as a node function in the case of full measurements. Unlike the additive model, we showed that the presence of a static component in the node functions does not affect the identifiability, allowing us to work with a more general class of functions. 
Then, we introduced the notion of generic identifiability for nonlinear functions and characterized the measure-zero set associated with the generic notion as a subspace of finite dimension. For analytic functions in DAGs, we provided a sufficient condition for identifiability in terms of vertex-disjoint paths that coincide with the identifiability conditions in the linear case. For the class of polynomial functions, we showed that if the vertex-disjoint path condition is not satisfied, then no DAG is identifiable.

A natural continuation of this work is to derive identifiability conditions for more general digraphs where loops are present, which implies a function $F_i$ that depends on an infinite number of excitation signals.

\bibliographystyle{IEEEtran}
\bibliography{arxiv}

\end{document}